\documentclass[11pt]{amsart}
\usepackage{amsmath,amsthm, amscd, amssymb, amsfonts}

\usepackage[]{color} 

\usepackage{graphicx} 

\usepackage{mathrsfs} 
\usepackage{figsize}
\usepackage{algorithm}
\usepackage{algpseudocode}
\usepackage{enumerate} 

\usepackage{float} 

\usepackage{tikz} 

\usepackage{tkz-berge} 
\usepackage{tkz-graph}
\usetikzlibrary{arrows}
\usetikzlibrary{calc}
\usetikzlibrary{matrix}
\usetikzlibrary{decorations.pathreplacing}  

\usepgflibrary{shapes.geometric} 
\usetikzlibrary[topaths]
\usetikzlibrary{positioning}
\usetikzlibrary{arrows}  
\usetikzlibrary{graphs,graphs.standard}

\usepackage{tabularx}

\theoremstyle{plain}
\newtheorem{theorem}{Theorem}[section]
\newtheorem{corollary}[theorem]{Corollary}
\newtheorem{proposition}[theorem]{Proposition}
\newtheorem{lemma}[theorem]{Lemma}
\newtheorem{definition}{Definition}[section]  
 
\newtheorem{remark}[theorem]{Remark}
\newtheorem{example}[theorem]{Example}

\usepackage[colorlinks]{hyperref}

\def\F{\mathbb{F}}

 \def\ad{\mathrm{ad}}
 \def\der{\mathrm{Der}}
 \def\ch{\mathrm{char}}

\def\sol{\mathrm{sol}}

\title[The solvable Graph]{The solvable Graph of a finite-dimensional Lie Algebra}

\thanks{...}

\author{David Towers}

 \address{Lancaster University, School of Mathematical Sciences, 
 }
  \email{d.towers@lancaster.ac.uk
 }

 \author{Ismael Gutierrez}
 \address{Universidad del Norte,  Departamento de Matem\'aticas y Estad\'istica, Km 5 via a Puerto Colombia, Barranquilla, Colombia.}
  \email{isgutier@uninorte.edu.co
 }

\author{Luis Fernandez}
\address{Universidad del Norte,  Departamento de Matem\'aticas y Estad\'istica, Km 5 via a Puerto Colombia, Barranquilla, Colombia.}
   \email{lfernandeze@uninorte.edu.co}

\subjclass[2010]{17B30, 17B45, 05C40, 05C25}
\keywords{Graphs associated with Lie Algebras; solvable Lie Algebra; solvabilizer}

\begin{document} 


\begin{abstract}
We introduce and investigate the solvable graph $\Gamma_\mathfrak{S}(L)$ of a finite-dimensional Lie algebra $L$ over a field $F$. The vertices are the elements outside the solvabilizer $\sol(L)$, and two vertices are adjacent whenever they generate a solvable subalgebra. After developing the basic properties of solvabilizers and $S$-Lie algebras, we establish divisibility conditions, coset decompositions, and degree constraints for solvable graphs. Explicit examples, such as $\mathfrak{sl}_2(\mathbb{F}_3)$, illustrate that solvable graphs may be non-connected, in sharp contrast with the group-theoretic setting. We further determine the degree sequences of $\Gamma_\mathfrak{S}(\mathfrak{gl}_2(\F_q))$ and $\Gamma_\mathfrak{S}(\mathfrak{sl}_2(\F_q))$, highlighting how spectral types of matrices dictate combinatorial patterns. An algorithmic framework based on GAP and SageMath is also provided for practical computations. Our results reveal both analogies and differences with the nilpotent graph of Lie algebras, and suggest that solvable graphs encode structural invariants in a genuinely new way. This work opens the door to a broader graphical approach to solvability in Lie theory.

\end{abstract}

\maketitle


\section{Introduction}

The systematic study of the solvable and non-solvable graphs of finite groups began in 2013, when Hai-Reuven introduced the non-solvable graph $\mathcal{S}_G$ of a group $G$ and proved that, for any non-solvable finite group, the induced subgraph on $G\setminus\sol(G)$ is never regular and has diameter at most $2$ \cite{HaiReuven}. Six years later, Bhowal, Nongsiang, and Nath investigated the complementary solvable graph $\Gamma_\mathcal{S}(G)$, showing that it is neither a tree nor $n$-partite, and can never be planar, toroidal, or projective. They also established a lower bound on its clique number in terms of the group’s \emph{solvability degree} \cite{Bhowal2019}. In 2020, Akbari \textit{et al.} proved that every finite simple group yields a \emph{split} solvable graph and that the degree pattern of $\Gamma_\mathcal{S}(G)$ can distinguish various classical groups \cite{Akbari2020}. A comprehensive monograph by Yadav and collaborators followed in 2021, elucidating the interplay between $\mathcal{S}_G$, the solvable radical $R(G)$, and invariably generating sets \cite{Yadav2021}. Most recently, Akbari embedded both $\Gamma_\mathcal{S}(G)$ and $\mathcal{S}_G$ into a broader ``supergraph'' framework, improving the universal diameter bound to $2$ for all non-solvable groups \cite{Akbari2024}.

These developments highlight a unifying theme: algebraic structure can often be encoded in the combinatorial properties of graphs attached to groups. In this spirit, it is natural to ask whether similar constructions are possible for Lie algebras. Although finite groups and finite-dimensional Lie algebras differ substantially in structure, both possess solvable radicals, centralizers, and natural generating subsets, suggesting that their corresponding ``solvable graphs'' may display analogous phenomena. 

More recently, attention has also turned to the Lie-theoretic counterpart of these constructions. In \cite{TGF}, the nilpotent graph of a finite-dimensional Lie algebra was introduced and systematically studied, establishing a close parallel with the solvable and non-solvable graphs of groups. It was shown there that the nilpotentizer behaves well with respect to direct sums, that the graph encodes structural properties such as connectivity and isolation of vertices, and that explicit computations for triangular matrix algebras reveal highly regular decompositions. This complementary perspective highlights the robustness of the graphical approach to solvability and nilpotency across different algebraic settings.

The present work initiates a systematic investigation of the \emph{solvable graph} of a finite-dimensional Lie algebra $L$, where each vertex represents an element of $L$ and adjacency encodes the failure of the subalgebra generated by two elements to be solvable. We establish several basic structural properties of these graphs, including divisibility conditions on the sizes of solvabilizers, constraints on degrees, and non-connectedness phenomena. Along the way, we demonstrate explicit examples for small Lie algebras such as $\mathfrak{sl}_2(\F_3)$, showing how their solvable graphs decompose into multiple components, and we provide algorithmic tools for computing such graphs in practice. 

Beyond these foundational results, we also prove that in characteristic zero a simple $S$-Lie algebra must be abelian of dimension one, and we formulate a strong divisibility conjecture supported by extensive computational evidence. Taken together, these results reveal both striking parallels and crucial differences with the group-theoretic case, and open the door to further exploration of solvable and non-solvable graphs in the Lie algebraic setting.

The present article should be viewed as complementary to our earlier work on the \emph{nilpotent graph} of a finite-dimensional Lie algebra \cite{TGF}. In both settings, one attaches a combinatorial structure to an algebra by considering two-generated subalgebras and by testing for structural properties such as nilpotency or solvability. The analogy is striking: in the nilpotent graph, connectivity phenomena and coset decompositions play a central role, while in the solvable graph, we observe parallel divisibility conditions and degree constraints. At the same time, there are key differences: for instance, solvable graphs may fail to be connected even in very small examples such as $\mathfrak{sl}_{2}(\F_3)$, and the hierarchy of divisibility relations for solvabilizers has no direct analogue in the nilpotent case. This illustrates that the two graph perspectives are genuinely complementary and that the solvable graph captures new aspects of Lie algebra structure.

In Section~2, we introduce the basic definitions of solvabilizers and the solvable graph of a Lie algebra, and we establish their first structural properties, such as coset decompositions and elementary divisibility results. 

Section~3 develops the general theory of $S$-Lie algebras, including constraints on centralizers and solvabilizers, and culminates with the proof that in characteristic zero every simple $S$-Lie algebra is one-dimensional abelian. 

In Section~4 we investigate the global properties of solvable graphs, including degree bounds, connectivity issues, and examples computed for small Lie algebras such as $\mathfrak{sl}_2(\F_3)$. 

Finally, Section~5 outlines a strong divisibility conjecture, supported by computational evidence, and discusses directions for further research on the interplay between Lie algebra structure and solvable graphs.

\section{Preliminaries}

\subsection*{Basics on Lie Algebras.}
Let $L$ be a finite-dimensional Lie algebra over a field $F$. A \emph{Lie algebra} is a vector space $L$ over $F$, equipped with a bilinear operation $[\cdot,\cdot]: L \times L \to L$, called the \emph{Lie bracket}, satisfying the following axioms:
\begin{enumerate}
    \item \textbf{Alternating property:} $[x,x] = 0$ \ $\forall  x \in L$;
    \item \textbf{Jacobi identity:} $[x,[y,z]] + [y,[z,x]] + [z,[x,y]] = 0$ $\forall x, y, z \in L$.
\end{enumerate}
These imply \textbf{anticommutativity}: $[x,y] = -[y,x]$ for all $x,y \in L$.

A Lie algebra is said to be \emph{abelian} if $[x,y] = 0$ for all $x,y \in L$. Every one-dimensional Lie algebra is abelian. The \emph{dimension} of a Lie algebra is its dimension as a vector space over $F$. A \emph{subalgebra} of $L$ is a subspace closed under the Lie bracket. If $A, B$ are subspaces of $L$, then $[A,B]$ denotes the subspace spanned by all $[a,b]$ with $a \in A$, $b \in B$. 

 Let $F$ be a field. Standard examples include: 
\begin{itemize}
    \item $\mathfrak{gl}_n(F)$, the Lie algebra of all $n \times n$ matrices over $F$ with bracket $[x,y] = xy - yx$;
   \item $\mathfrak{t}_n(F)$, the subalgebra of upper triangular matrices;
    \item $\mathfrak{sl}_n(F)$, the subalgebra of traceless $n \times n$ matrices over $F$ .
   \item $\mathfrak{so}_n(F)$, the subalgebra of skew-symmetric $n \times n$  matrices relative to the standard quadratic form.
    \item $\mathfrak{sp}_{2m}(F)$, the symplectic Lie algebra over $F$ .
\end{itemize}

A subspace $I$ of $L$ is called an \emph{ideal} if $[x,y] \in I$ for all $x \in I$, $y \in L$, that is, $[I,L] \subseteq I$. Intersections, sums, and Lie products of ideals are again ideals.

The \emph{derived series} of $L$ is defined recursively by
\begin{align*}
    L^{(0)} &= L, \\
    L^{(k+1)} &= [L^{(k)}, L^{(k)}].
\end{align*}
The \emph{lower central series} is given by
\begin{align*}
    L^1 &= L, \\
    L^{k+1} &= [L^k, L].
\end{align*}
A Lie algebra $L$ is called \emph{solvable} if $L^{(r)} = 0$ for some $r$, and \emph{nilpotent} if $L^k = 0$ for some $k$. Every nilpotent Lie algebra is solvable, but the converse is not true.

Examples: Let $F$ be a field.
\begin{itemize}
    \item If $\ch(F) \neq 2$ then $\mathfrak{gl}_n(F)$  is non solvable;
    \item If $\ch(F) \neq 2$ then $\mathfrak{sl}_n(F)$ is non solvable;
    \item $\mathfrak{t}_n(F)$ is solvable but not nilpotent;
   \item $\mathfrak{so}_n(F)$ is non solvable;
    \item $\mathfrak{sp}_{2m}(F)$ is non solvable.
\end{itemize}

For $x \in L$, define the \emph{adjoint map} $\ad\, x : L \to L$ by $\ad\, x(y) = [x,y]$. The map $\ad: L \longrightarrow \der(L)$, where $\der(L)$ is the set of derivations of $L$, is called the \emph{adjoint representation} of $L$. The \emph{centralizer} of $x$ in $L$, denoted as $C_L(x)$ or simply $C(x)$ when $L$ is clear from context, is defined as: 
\[C_L(x) = \{y\in L\mid [x,y] = 0 \}\] 
where $[x,y]$ denotes the Lie bracket of $x$ and $y$.

The \emph{center} of $L$ is $Z(L) = \{x \in L \mid [x,L] = 0\}$. Define the upper central series recursively as
\begin{align*}
    Z_0(L) &= 0, \\
    Z_i(L)/Z_{i-1}(L) &= Z(L/Z_{i-1}(L)), \quad i \geq 1.
\end{align*}
The union $\bigcup_{i \geq 0} Z_i(L)$ is called the \emph{hypercentre}, denoted $Z^*(L)$. For finite-dimensional $L$, this series terminates.




\subsection*{Basics on Graph Theory.}

We begin by recalling some basic notions from graph theory that will be used throughout this paper. A \emph{graph} $\Gamma$ is an ordered pair $(V,E)$, where $V$ is the set of vertices and $E$ is the set of unordered pairs of distinct elements of $V$, called edges. If ${u,v} \in E$, then $u$ and $v$ are said to be \emph{adjacent}, or \emph{neighbors}. 
A graph is called \emph{simple} if it has neither loops nor multiple edges; in what follows, we shall always assume graphs to be simple. The \emph{degree} of a vertex $v$, denoted $\deg(v)$, is the number of vertices adjacent to $v$. A graph is said to be \emph{$k$-regular} if every vertex has degree $k$, i.e., $\deg(v)=k$ for all $v \in V$.

A graph is \emph{complete} if every pair of distinct vertices is adjacent; the complete graph on $n$ vertices is denoted by $K_n$. A graph is \emph{connected} if there exists a path between any two of its vertices, and a \emph{connected component} is a maximal connected subgraph. 
The \emph{complement} of a graph $\Gamma$, denoted by $\bar{\Gamma}$, is defined on the same vertex set $V$ with edges ${u,v}$ whenever $u$ and $v$ are not adjacent in $\Gamma$.

For undefined terms and further background in graph theory, the reader is referred to \cite{Diestel} and \cite{Nora}.

\section{The solvabilizer of a Lie algebra}
 
In this paper, $\mathfrak{S}$ and $\mathfrak{S}_c$ denote the class of all finite-dimensional solvable Lie algebras and the class of all finite-dimensional solvable Lie algebras of the solvability degree $c$, respectively.

\begin{definition}
Let $L$ be a finite-dimensional Lie algebra over a field $F$. 
\begin{enumerate}
    \item For $h\in L$ we define  the \emph{solvabilizer} of $h$ in $L$ as follows:
\begin{equation*}
\sol_L(h) := \{x\in L\mid  \langle h,x\rangle \ \text{is a solvable Lie subalgebra of} \  L \}.
\end{equation*}

\item The solvabilizer of $L$ is defined by:
\begin{equation*}
\sol(L) := \{x\in L\mid  \langle h,x\rangle \ \text{is solvable for all} \ h\in L\}.
\end{equation*}
\end{enumerate}
\end{definition}

Let $L$ be a finite-dimensional Lie algebra defined over a field $F$. We denote with $R(L)$ the maximal solvable ideal
of $L$, and we call it \emph{the solvable radical of $L$}. 

\begin{example}
\begin{enumerate}
 \item $\sol(\mathfrak{sl}_2(F_2))= \mathfrak{sl}_2(F_2)$. 
 \item $\sol(\mathfrak{gl}_2(F_2))= \mathfrak{gl}_2(F_2)$. 
 \item If $\ch(F) \neq 2$, then $\sol(\mathfrak{sl}_2(F))= 0$.
 \item If $\ch(F) \neq 2$, then $\sol(\mathfrak{gl}_2(F)) = Z(\mathfrak{gl}_2(F)) \neq 0$.
\end{enumerate}
\end{example}

\begin{remark}\label{sl2(3)-basis}
Some immediate properties:
\begin{enumerate}
\item $\langle h\rangle \subseteq \sol_L(h)$,

\item $\sol(L) = \bigcap_{h\in L} \sol_L(h)$,

\item $\sol_L(h)$ is in general not a subalgebra of $L$.  Let  $F$ be any field with $\mathrm{char} F\neq 2$ and $L=\mathfrak{sl}_{2}(F)$. A basis for $L$ is $B= \{e, f, h\}$, where:
\[  e \;=\;   \begin{pmatrix} 
0 & 1 \\ 0 & 0 
\end{pmatrix},
  \quad
  f \;=\;   \begin{pmatrix} 
  0 & 0 \\ 1 & 0 
  \end{pmatrix},
  \quad
  h \;=\;   \begin{pmatrix} 
  1 & 0 \\ 0 & -1 
  \end{pmatrix}.\]
Their Lie brackets are $[h,e]=2e,\; [h,f]=-2f,\; [e,f]=h$. Note that $\langle h, e\rangle$ and $\langle h, f\rangle$ are both two-dimensional and solvable subalgebras of $L$.  Hence $e, f\in\sol_L(h)$. We now define $s:=e+f$ and $t:=[h,s]=2(e-f)$.
Then $[s,t]=[e+f,e-f]=-2h$, so the set $\{h,s,t\}$ spans the whole algebra $L$. Consequently $\langle h, s \rangle = L$ is not solvable, whence $s\notin\sol_L(h)$. Therefore, $\sol_L(h)$ is not closed under addition; hence, it is not a Lie subalgebra of $L$.
\end{enumerate}
\end{remark}


\begin{theorem}\cite[Theorem 2.1]{Thompson}\label{Thompson}
Let $L$ be a finite-dimensional Lie algebra defined over a field $F$ of characteristic zero. Then  $R(L)=\sol(L)$.
\end{theorem}


In positive characteristic, this equality may fail, as the following example shows.

\begin{example}\label{char2}
Let $F$ be a field of characteristic two and let $L$ be the Lie algebra over $F$ generated by the matrices $$a=\left( \begin{array}{ccc} 0&0&0\\
0&1&0\\
0&0&1
\end{array}\right),  b=\left( \begin{array}{ccc} 0&1&0\\
0&0&0\\
1&0&0
\end{array}\right),  c=\left( \begin{array}{ccc} 0&0&1\\
1&0&0\\
0&0&0
\end{array}\right)$$ relative to the commutation operator. Then $[a,b]=b$, $[a,c]=c$ and $[b,c]=a$.

\medskip

\noindent {\bf Claim 1: $L$ is simple.}
\medskip

Let $I$ be an ideal of $L$ and let $\alpha a+\beta b+\gamma c \in I$. If $a\in I$, then $I=L$, so suppose $a\notin I$.. Now $[a,\alpha a+\beta b+\gamma c]]=\beta b+\gamma c\in I,$
so $\alpha x\in I$ and $\alpha=0$. Also, $[b,\beta b+\gamma c]=\gamma a\in I$, so $\gamma=0$. Similarly, $\beta=0$.
\medskip

\noindent {\bf Claim 2: $a\in$ sol($L$)}
\medskip

Clearly, $\langle a,\alpha a+\beta b+\gamma c\rangle=\langle a,\beta b+\gamma c\rangle=Fa+F(\beta b+\gamma c)$, which is solvable. 
\medskip

It follows that, if sol($L$) is an ideal of $L$, we must have that sol($L$) $=L$. But $b,c \notin$ sol($L$), since $\langle b,c \rangle=L$, which is not solvable. 
\end{example}

However, we have the next result.

\begin{definition}
A Lie algebra is called {\it classical semisimple} if it is a direct sum of classical simple Lie algebras.    
\end{definition}

\begin{theorem} Let $L$ be a finite-dimensional Lie algebra over an infinite field $F$ of characteristic $p>3$, let $R(L)$ be the solvable radical of $L$ and suppose that $L/R(L)$ is classical semisimple. Then $R(L)$ coincides with the set of elements $y\in L$ with the property that $\langle x,y\rangle$ is solvable for every $x\in L$.
\end{theorem}
\begin{proof} If $y\in R(L)$, then, for every $x\in L$, $\langle x,y\rangle = \langle x,y\rangle\cap R(L) + Fx$, which is solvable. So suppose that $y\notin R(L)$. We need to show that there is an $x\in L$ such that $\langle x,y\rangle$ is not solvable. By factoring out $R(L)$, we may assume that $L$ is semisimple. Let $L=S_1\oplus \ldots \oplus S_n$, where $S_i$ is classical simple and put $y=s_1+ \ldots + s_n$, where $s_i\in S_i$. Then there is an element $s\in S_1$ such that $\langle s,s_1\rangle=S_1$, by \cite[Theorem B and section 1.2.2]{bois}, and $\langle x,s\rangle$ is not solvable.
\end{proof}

\begin{corollary}
 With the same hypothesis, $sol(L)=R(L)$.    
\end{corollary}

Investigating $\sol(L)$ when $L$ is of Caratan type remains an open problem.

\subsection*{Properties of solvabilizers}
If $x\in L$, then $\sol_L(x)$ is the union of all maximal solvable subalgebras containing $x$.
\par

We will need the following result from linear algebra.

\begin{lemma}\label{subspaces}  Let $B_1, \ldots, B_n$ be subspaces of a vector space $V$ such that $B_i\not\subseteq \cup_{j\neq i} B_j$. If $|F|>n-1$, then $\cup_{i=1}^nB_i$ is not a subspace
\end{lemma}
\begin{proof} Suppose that $B=\cup_{i=1}^nB_i$ is a subspace. Let $x\in B_1, x\notin \cup_{j>1}B_j,y\notin B_1$ and put $C=y+Fx$. Then $|C\cap B_1|=0$. Also, $|C\cap B_j|\leq 1$ for $j>1$. For, if $y+\alpha x, y+\beta x\in C\cap B_j$ (with $\alpha\neq \beta$), then $x\in B_j$, a contradiction.  Hence $|C|\leq n-1$.
\par

Let $\theta: F \rightarrow C: \alpha \mapsto y+\alpha x$. Then $\theta$ is injective, whence $|F|=|C|\leq n-1$.
\end{proof}

\begin{lemma}\label{solvabilizer1}
Let $L$ be a Lie algebra of dimension $n$ over a field $F$ with $|F|> n-2$, and let $x\in L$. Then the following statements are equivalent:
\begin{enumerate}[\rm(i)]
  \item $\sol_L(x)$ is a \emph{subspace} of $L$;
  \item $\sol_L(x)$ is a \emph{maximal solvable subalgebra} of $L$;
  \item $x$ lies in a unique maximal solvable subalgebra of $L$.
\end{enumerate}
\end{lemma}

\begin{proof}
First, we claim that
\begin{equation}\label{sol1}
\sol_L(x) = \bigcup\big\{M\le L \mid M \text{ maximal solvable and } x\in M \big\}.
\end{equation}
In fact, let $y\in \sol_L(x)$. Then $\langle x,y\rangle$ is solvable, hence it is contained in some
maximal solvable subalgebra $M$ with $x,y\in M$; thus $y$ lies in the right side of \eqref{sol1}.

Conversely, if $y$ belongs to a maximal solvable subalgebra $M$ containing $x$, the subalgebra $\langle x, y\rangle\subseteq M$ is solvable, so $y\in\sol_L(x)$. This proves \eqref{sol1}
\medskip

\noindent\rm{(ii) $\Rightarrow$ (i).} Any maximal solvable subalgebra is of course a subspace.
\medskip

\noindent\rm{(iii) $\Rightarrow$ (ii).} If there is a unique maximal solvable subalgebra $M$
containing $x$, then by \eqref{sol1} we have $\sol_L(x)=M$, which is maximal solvable by hypothesis.
\medskip

\noindent\rm{(ii) $\Rightarrow$ (iii).} Suppose $M:=\sol_L(x)$ is a maximal solvable subalgebra of $L$. Let now $N$ be another maximal solvable subalgebra $N\neq M$ contained $x$, then $N \subseteq \sol_L(x) = M$ by \eqref{sol1}; maximality forces
$N=M$, a contradiction.  Hence $x$ lies in exactly one maximal
solvable subalgebra.
\medskip

\noindent\rm{(i) $\Rightarrow$ (iii).}  Assume $\sol_L(x)$ is a subspace of $L$.  Suppose further that $x$ lies in $k>1$ maximal solvable subalgebras $M_1, \ldots, M_k$ with $M_i \not\subseteq \cup_{j\neq i} M_j$. Since $\dim M_i\geq 2$, $n\geq \dim \cup_{i=1}^kM_i\geq k+1$. By Lemma \ref{solvabilizer1}, $|F|\leq k-1\leq n-2$, a contradiction. Therefore, there is exactly one maximal solvable subalgebra containing $x$.
\end{proof}

\begin{lemma}\label{sol-properties}
Let $L$ be a Lie algebra, $J$ an ideal of $L$ and $x,y\in L$. Then
\begin{enumerate}[(i)]
\item  $\sol(L)\subseteq \sol_L(x)$;
\item  $\frac{\sol_L(x)+J}{J}\subseteq \sol_L(x+J)$;
\item  $\sol_{L/J}(x+J)=\frac{\sol_L(x)}{J}$ whenever $J\subseteq R(L)$;
\item  the solvabilizer of every element of $L$ is a subalgebra if and only if the same is true for $L/J$ for some ideal $J$ of $L$ with $J\subseteq R(L)$;
\item  if $\sol_L(x)$ is a subalgebra of $L$ and $y\in \sol_L(x)$, then $\sol_L(y)=\sol_L(x)$;
\item  if $\theta$ is an automorphism of $L$, then $\theta(\sol_L(x))=\sol_L(\theta(x))$.
\end{enumerate}
\end{lemma}

\begin{proof} 
All statements follow immediately from the definition of $\sol_A(B)$ and elementary set-theoretic manipulations.
\end{proof}

This lemma collects a number of functorial properties of solvabilizers, showing how they behave under quotients, ideals, and automorphisms. In particular, (iii) and (iv) explain how passing to quotients with solvable radical preserves the subalgebra structure of solvabilizers, which will be useful later when analysing $S$-Lie algebras.

\begin{definition}\label{Sol-subset}
Let $L$ be a finite-dimensional Lie algebra over a field $F$,
and let $A, B\subseteq L$ be non-empty subsets.
The \emph{solvabilizer of $B$ with respect to $A$} is defined as follows: 
\[\sol_{A}(B)\;:=\;\bigl\{\,a\in A \mid \langle a,b\rangle
        \text{ is a solvable Lie subalgebra of }L
        \text{ for all }b\in B\bigr\}.\]
We put  
$\sol_{\emptyset}(B) := \emptyset$ and $\sol_{A}(\emptyset):=A$.
Note that $\sol_{A}(x) =\sol_{A}(\{x\})$ for $x\in L$, and $\sol(L) = \sol_{L}(L)$.
\end{definition}

\begin{lemma}\label{Lie-2.9}
Let $L$ be a finite–dimensional Lie algebra over a field $F$.
For non-empty subsets $A, B, C\subseteq L$ hold:
\begin{enumerate}
\item If $A\subseteq B$, then $\sol_{A}(C)\subseteq\sol_{B}(C)$  
      and $\sol_{C}(B)\subseteq\sol_{C}(A)$.

\item $\sol_{A}\!\bigl(\sol_{B}(A)\bigr)=A$.

\item If $A\subseteq B$, then
      $\sol_{A}(C)=A\cap\sol_{B}(C)$.

\item $\sol_{C}(A\cup B)=\sol_{C}(A)\cap\sol_{C}(B)$ and  
      $\sol_{C}(A\cap B)\supseteq\sol_{C}(A)\cup\sol_{C}(B)$.

\item $\displaystyle
      \sol_{A}(B)=\bigcap_{x\in B}\sol_{A}(x).$
      In particular,
      $\displaystyle
      \sol(L)=\bigcap_{x\in L}\sol_L(x).$
\end{enumerate}
\end{lemma}

\begin{proof} 
These are straightforward. 
\end{proof}

\begin{lemma}\label{Lie-2.6}
Let $L$ be a finite-dimensional Lie algebra over a field $F$ of characteristic zero,
let $H\le L$ be a solvable subalgebra, and let $s\in\sol(L)$.
Then the Lie subalgebra $\langle H, s\rangle$ is itself solvable.
\end{lemma}

\begin{proof}
Set $K:=H+R(L)\subseteq L$. Since $R(L)$ is a solvable ideal of $L$ and $K/R(L)\cong H/(H\cap R(L))$ is a homomorphic image of the solvable algebra $H$, it follows that $K$ is solvable. 

Next, using Theorem \ref{Thompson}, we have $s\in \sol(L)=R(L)$, and as $R(L)$ is an ideal we have
$[s,h]\in [R(L),H]\subseteq R(L)\subseteq K$ for every $h\in H$. Therefore every
iterated Lie bracket formed from elements of $H\cup\{s\}$ lies in $K$, so
$\langle H,s\rangle\subseteq K$. Since $K$ is solvable, it follows that $\langle H,s\rangle$ is
solvable.
\end{proof}

\begin{lemma}\label{Lie-2.7}
Let $L$ be a finite-dimensional Lie algebra over a field $F$. Then, for every $x\in L$,
\[\sol(L) + \sol_L(x) = \sol_L(x).\]
\end{lemma}

\begin{proof}
It follows immediately from Lemma \ref{sol-properties} (i).


\end{proof}

\begin{definition}\label{S-Lie}
Let $L$ be a finite–dimensional Lie algebra over a field $F$. We say that $L$ is an S–Lie algebra if $\sol_L(x)$ is a Lie subalgebra of $L$ for every $x\in L$.
\end{definition}

\begin{example}
Every solvable finite–dimensional Lie algebra over a field $F$ is an S-Lie algebra.
\end{example}





\begin{lemma}\label{Lie-2.10}
Let $L$ be a finite–dimensional S-Lie algebra. Then, for any Lie subalgebra $H\leq L$ and any subset $A\subseteq L$ (possibly empty) $\sol_{H}(A)$ is a Lie subalgebra of $L$ (and in fact of $H$).
\end{lemma}

\begin{proof}
By definition,
\[\sol_H(A) = H\cap \sol_L(A) = H\cap \bigcap_{a\in A} \sol_L(a).\]
Hence $\sol_{H}(A)$ is an intersection of Lie subalgebras of $L$ and then a Lie subalgebra of $L$. For the special case $A=\varnothing$, our convention $\sol_{H}(\varnothing)=H$ already gives a subalgebra.
\end{proof}




\begin{lemma}\label{rad} 
Let $L$ be a Lie algebra over any field $F$ with radical $R$. Then $R\subseteq \sol(L)$ and $\sol(L/R) = \sol(L)/R$.
\end{lemma}

\begin{proof} 
Let $x\in R$ and $y\in L$. Then $\langle x,y\rangle = \langle x,y\rangle\cap R+Fy$, so $\langle x,y\rangle$ has a solvable ideal of codimension at most one and thus is solvable. Hence $x\in \sol(L)$.
\par

Let $x+R\in \sol((L/R)$, $y\in L$. Then $\langle x,y\rangle+R=\langle x+R,y+R\rangle$ is solvable. But then $\langle x,y\rangle$ is solvable, so $x\in \sol(L)$. The reverse inclusion is clear.
\end{proof}

The next theorem shows that, outside characteristic $2$, the only ideals that can lie inside the solvabilizer of $L$ are those already contained in the solvable radical. In other words, the solvabilizer does not introduce new solvable Ideals beyond the radical, which reinforces the structural role of $R(L)$.

\begin{theorem} 
Let $L$ be a Lie algebra over an infinite field $F$ of characteristic greater than $2$ with radical $R$. If $N$ is an ideal of $L$ with $N\subseteq \sol(L)$ then $N\subseteq R$.
\end{theorem}

\begin{proof} 
 Since the set of elements which generate a solvable subalgebra is closed in the Zariski topology and $L$ is dense in some algebraic closure of $F$, we may assume that $F$ is algebraically closed. Suppose that $N\not\subseteq R$. Then $(N+R)/R$ is an ideal of $L/R$ inside $\sol(L/R)$, so it suffices to assume that $L$ is semisimple and to show that $N=0$. If $N\neq 0$ there will be a minimal ideal $I$ of $L$ inside $\sol(L)$ and $I$ contains a unique maximal ideal $J$, by \cite[Theorem 9.3]{block}, so $I/J$ is simple. Moreover, $I/J$ has a subalgebra $S/J$ with quotient $\frac{S/J}{K/J}\cong S/K\cong sl_2(F)$, by \cite[Lemma 3.2]{gknp}. Let $x$ be such that $x+K\in S/K$. Then there exists $y+K$ such that $\langle x,y \rangle+K=\langle x+K,y+K\rangle =S/T$. But $\langle x,y\rangle$ cannot be solvable, so $x\notin \sol(L)$, a contradiction. Hence $N=0$ as required.
\end{proof}

Next Lemma shows that $\sol_L(x)$ depends only on the one-dimensional subalgebra generated by $x$; in particular, it is unchanged when $x$ is rescaled.

\begin{lemma}\label{Lie-2.11}
Let $L$ be a finite–dimensional Lie algebra over a field $F$.
Let $x,y\in L$, let $N$ be an ideal of $L$ with $N\subseteq \sol(L)$, and let $\varphi\in\operatorname{aut}(L)$ be any Lie–algebra automorphism of $L$.
\begin{enumerate}
\item If $\langle x\rangle=\langle y\rangle$, then $\sol_L(x) = \sol_L(y)$.

\item $\sol_L(\varphi(x)) = \varphi(\sol_L(x))$.

\item Let $\pi : L \longrightarrow L/N$ be the canonical projection. Then \[\sol_{L/N}(\pi(x)) = \pi(\sol_L(x)) = \sol_L(x)/N.\]
\end{enumerate}
\end{lemma}

\begin{proof}
(1)\; Suppose $\langle x\rangle=\langle y\rangle$. Then for every
$a\in L$ we have $\langle a, x\rangle \subseteq \langle a, y\rangle$ and $\langle a, y\rangle \subseteq \langle a, x\rangle$. That is, $\langle a, x\rangle = \langle a, y\rangle$ for all $a\in L$. Therefore, $\sol_L(x) = \sol_L(y)$. \smallskip

\noindent (2)\; Let $z\in L$.  Then
\begin{align*}
z \in \sol_L(\varphi(x)) & \Longleftrightarrow\; \langle z, \varphi(x)\rangle \text{ is solvable} \\
& \Longleftrightarrow\;  \langle\varphi^{-1}(z),x\rangle \text{ is solvable}\\
& \Longleftrightarrow\; \varphi^{-1}(z)\in\sol_L(x)\\
& \Longleftrightarrow\; z\in\varphi(\sol_L(x)).
\end{align*}

\smallskip
\noindent (3)\; If $a\in\sol_L(x)$ then $\langle a,x\rangle$ is solvable; and its image $\pi(\langle a,x\rangle)=\langle \pi(a),\pi(x)\rangle$ is solvable because $\pi$ is a homomorphism.
Thus $\pi(a)\in\sol_{\,L/N}\bigl(\pi(x)\bigr)$ and $\pi(\sol_L(x)) \subseteq \sol_{L/N}(\pi(x))$.

Conversely, let $b+N\in\sol_{L/N}(\pi(x))$. Then the subalgebra $\langle b+N, x+N\rangle$ of $L/N$ is solvable, so $\langle b, x\rangle+N$ is solvable; hence $\langle b,x\rangle$ itself is solvable. Therefore $b\in\sol_L(x)$ and $b+N\in \pi(\sol_L(x))$. So $\sol_{L/N}(\pi(x))=\pi(\sol_L(x))$, which is the same as $\sol_L(x)/N$.

This result shows that solvabilizers only depend on the one-dimensional subalgebra $\langle x \rangle$, and are stable under automorphisms and quotients. Hence, the notion of solvabilizer is robust under standard structural operations 
on Lie algebras.



\end{proof}

\begin{lemma}\label{Lie-2.12}
Let $L$ be a finite–dimensional Lie algebra over a field $F$. Then, for every $x\in L$, $\sol_L(x)$ is a disjoint union of some cosets of $Fx$. More precisely, for every $a\in \sol_L(x)$ one has $a + F x \subseteq \sol_L(x)$, and hence
\[\sol_L(x) = \bigsqcup_{a\in\mathcal R}\,(a+F x)\]
for some set $\mathcal R\subseteq \sol_L(x)$ of representatives of the
cosets of $F x$ contained in $\sol_L(x)$.
\end{lemma}

\begin{proof}
Fix $x\in L$ and $a\in \sol_L(x)$. For any $t\in F$ we have $\langle a+ t x, x\rangle = \langle a, x\rangle$.
As $\langle a,x\rangle$ is solvable by the choice of $a$, the same holds for $\langle a+t x,x\rangle$, so $a+t x\in \sol_L(x)$. Thus $a+F x\subseteq \sol_L(x)$.

Cosets $a+F x$ in the additive group $L$ are either equal or disjoint, hence $\sol_L(x)$ is a disjoint union of those cosets that it contains; choose one representative $a$ from each such coset to form $\mathcal{R}$.
\end{proof}

\begin{corollary}\label{divisible-by-q}
If $x\in L$ and $F$ is a finite field with $|F|=q$, then $q$ divides $|\sol_L(x)|$. 
\end{corollary}

\begin{corollary}\label{partition-by-SolL}
Let $L$ be a finite–dimensional Lie algebra over a field $F$. Assume that $\sol(L)$ is an additive subgroup of $L$.
Then for every $x\in L$,
\[\sol_L(x) = \bigsqcup_{a\in\mathcal R_x} (a+\sol(L))\]
for some set of representatives $\mathcal R_x\subseteq \sol_L(x)$.
Consequently, if $F$ is finite, $|\sol(L)| \mid |\sol_L(x)|$ and $\dim\sol(L) \leq \dim\sol_L(x)$ if each $\sol_L(x)$ is a subspace of $L$.
\end{corollary}

\begin{proof}
Lemma \ref{Lie-2.7} gives $\sol(L)+\sol_L(x)=\sol_L(x)$; if $\sol(L)$ is an additive subgroup, $\sol_L(x)$ is a union of cosets $a+\sol(L)$,
pairwise disjoint. The divisibility and the dimension inequality follow from counting cosets. 
\end{proof}

\begin{proposition}\label{Lie-2.13}
Let $L$ be a finite–dimensional S-Lie algebra over a field $F$, and let $x\in L$. Then:
\begin{enumerate}
\item $C_L(x)\subseteq \sol_L(x)$, $\dim C_L(x)\le\dim \sol_L(x)$.  In particular, if $F$ is finite, then $|C_L(x)|\le |\sol_L(x)|$.

\item There exists $\mathcal R_x\subseteq \sol_L(x)$ such that
\[\sol_L(x) = \bigsqcup_{a\in\mathcal R_x} (a+C_L(x)),\]      
and consequently if $F$ is finite $|C_L(x)|$ divides $|\sol_L(x)|$. 
\end{enumerate}
\end{proposition}

\begin{proof}
(1) If $c\in C_L(x)$ then $[c,x]=0$, so $\langle c,x\rangle$ is abelian and hence soluble; thus $c\in \sol_L(x)$. Since $L$ is an S-Lie algebra, the rest is clear. 

(2)  Since $C_L(x)\subseteq \sol_L(x)$, the standard coset decomposition of an additive group by a subgroup applies:
\[\sol_L(x) = \bigsqcup_{a\in\mathcal R_x} (a+C_L(x)).\]   
The counting statements are immediate.
\end{proof}

\begin{remark}
Taken together, the last four results reveal a layered coset structure for every solvabilizer $\sol_L(x)$. At the most elementary level, $\sol_L(x)$ decomposes into cosets of the one-dimensional subspace $Fx$, forcing divisibility of $|\sol_L(x)|$ by $|F|$ in the finite case. When $\sol(L)$ is additive, $\sol_L(x)$ refines further into cosets of $\sol(L)$, 
producing a stronger divisibility condition by $|\sol(L)|$. Finally, in the setting of $S$--Lie algebras, $\sol_L(x)$ admits the even finer decomposition into cosets of the centralizer $C_L(x)$, so that $|C_L(x)|$ divides $|\sol_L(x)|$. 
This hierarchy can be summarized as $|F|$ divides $|\sol_L(x)|$, $|\sol(L)|$ divides $|\sol_L(x)|$ and $|C_L(x)|$ divides $|\sol_L(x)|$, depending on the structural assumptions imposed on $L$. 
\end{remark}


\textbf{Conjecture 1.}\label{strong-div}
Let $L$ be a finite–dimensional Lie algebra over a finite field $F$. Then  $|L|$ divides $\sum_{x\in L} |\sol_L(x)|$. Equivalently, the average size $\frac{1}{|L|}\sum_{x\in L} |\sol_L(x)|$ is an integer. We computed $\sum_{x\in L}|\sol_L(x)|$ exhaustively in the following cases; in each case the divisibility by $|L|$ holds:

\begin{table}[H]
    \centering
    \begin{tabular}{c|c|c|c}
L & $|L|$ & $\sum_{x\in L}|\sol_L(x)|$ & \text{quotient} \\
\hline
$\mathfrak{sl}_2(\F_3)$ & 27  & 297   & 11 \\
$\mathfrak{sl}_2(\F_5)$ & 125 & 3625  & 29 \\
$\mathfrak{gl}_2(\F_3)$ & 81  & 2673  & 33 \\
    \end{tabular}
    \caption{Illustration of Conjecture 1}
    \label{placeholder}
\end{table}

These instances suggest that Conjecture 1 may hold in full generality.



\begin{lemma}\label{Lie-2.20}
Let $L$ be a finite–dimensional Lie algebra over a field $F$, and let $N\lhd L$ be an ideal with $N\subseteq \sol(L)$. Assume $N$ is solvable (This is automatic in characteristic $0$). Then $L$
is an  S–Lie algebra if and only if $L/N$ is an S–Lie algebra.
\end{lemma}

\begin{proof}
Suppose first that $L$ is an $S$-Lie algebra and let $x+N,y+N\in L/N$. Then $\langle x+N,y+N\rangle =\langle x,y\rangle +N$, which is solvable since $\langle x,y\rangle$ and $N$ are both solvable.
\par

Next, suppose that $L/N$ is an $S$-Lie algebra and let $x,y\in L$. Then $(\langle x,y\rangle +N)/N$ is solvable, whence $\langle x,y\rangle +N$ and so $\langle x,y\rangle$ is solvable
\end{proof}

\begin{proposition}\label{Lie-2.22}
Let $L$ be a finite–dimensional Lie algebra over a field $F$. Then
\begin{enumerate}
\item If $L$ is solvable, then $L$ is an $S$–Lie algebra; in fact $\sol_L(x)=L$ for every $x\in L$.
\item Conversely, assume that $F$ is algebraically closed of characteristic zero. If $L$ is an $S$–Lie algebra, then $L$ is solvable.
\end{enumerate}
In particular, under the standing hypothesis in \textup{(2)} we have: $L$ is solvable if and only if $L$ is an $S$–Lie algebra.
\end{proposition}

\begin{proof}
(1) is clear, so assume that $F$ is algebraically closed of characteristic zero and that $L$ is an $S$-Lie algebra. Suppose first that $L$ is semisimple. By general theory, it is clear that $L$ has more than one Borel subalgebra. Let $B$, $B'$ be two different Borel subalgebras. A standard Borel subalgebra has dimension greater than $\frac{1}{2}\dim L$ and all Borel subalgebras are conjugate. So, since $ \dim (B+B')=\dim B+\dim B'-\dim B\cap B'$, we have $B\cap B'\neq 0$. Let $x\in B\cap B'$. Then $sol_L(x)$ is not a subalgebra, by Lemma \ref{solvabilizer1} and the fact that the field is infinite. Thus $L$ is not an $S$-Lie algebra.
\par

Suppose now that $L$ is not semisimple nor solvable, so $L=S\dot{+} R$ where $R$ is the radical and $S$ is a Levi subalgebra. Then, if $x\in R$, any subalgebra of the form $R+B$, where $B$ is a Borel subalgebra of $S$ is a maximal solvable subalgebra containing $x$. As there is more than one such subalgebra, $L$ is not an $S$-Lie algebra. 
\end{proof}

Note that the assumption of algebraic closure is necessary in the above result, as the next example shows.

\begin{example}  Let $L$ be any three-dimensional non-split simple Lie algebra. Then all of the subalgebras are one-dimensional, so $sol_L(x) = Fx$ for all $x \in L$. Hence, these are $S$-Lie algebras.
\end{example}

Simple $S$-Lie algebras can also exist over a finite field, as is shown next.

\begin{example}
Let $L$ be the Lie algebra given in Example \ref{char2} over the field $F=\F_2$. Then  $sol_L(a) = L$, $sol_L(b) = Fa + Fb$, $sol_L(c) = Fa + Fc$, $sol_L(a+b) = Fa + Fb$, $sol_L(a+c) = Fa + Fc$, $sol_L(b+c) = Fa + F(b+c)$, $sol_L(a+b+c) = Fa + F(b+c)$, all of which are subalgebras. Hence, this is an $S$-Lie algebra.
\end{example}

However, it remains an open question as to whether $S$-Lie algebras over a finite field of characteristic $p>2$ are always solvable.





\section{The solvable graph of a Lie algebra}

\begin{definition}
Let $L$ be a finite-dimensional non-solvable Lie algebra. The solvable graph of $L$, denoted by $\Gamma_{\mathfrak{S}}(L)$, is a simple undirected graph whose set of vertexes is $L\setminus \sol(L)$, and two vertices $x$ and $y$ are adjacent if and only if $\langle x, y\rangle$ is a solvable subalgebra of $L$.    
\end{definition}

\begin{example}
The solvable graphs of $\mathfrak{gl}_2(\F_2)$ and $\mathfrak{sl}_2(\F_2)$ are empty.
\end{example}

\begin{example}\label{Example4.2}
The smallest non-empty instance of the solvable graph arises from $\mathfrak{sl}_2(\F_3)$. Since $\sol(\mathfrak{sl}_2(F))=0$ whenever $\ch(F)\neq 2$, all non-zero elements of $\mathfrak{sl}_2(\F_3)$ become vertices of $\Gamma_\mathfrak{S}(\mathfrak{sl}_2(\F_3))$, giving a graph of order $26$. A computation with SageMath confirms that the graph has size $109$. It has twelve vertices of degree $13$, eight of degree $7$, and six of degree $1$. The three degree classes correspond to the three spectral types of matrices in $\mathfrak{sl}_2(\F_3)$ (no eigenvalues, one eigenvalue, two eigenvalues). Let $B=\{f, h, e\}$ the basis of $\mathfrak{sl}_2(\F_3)$ described in Remark \ref{sl2(3)-basis}.  For instance, $f=\bigl(\begin{smallmatrix}0&0\\ 1&0\end{smallmatrix}\bigr)$ has degree $7$, $h=\bigl(\begin{smallmatrix}1&0\\ 0&-1\end{smallmatrix}\bigr)$ has degree $13$, while $e+f+h$ has degree $1$, its only neighbour being $2(e+f+h)$. This corroborates, in the smallest case, the description in Theorem~\ref{sl2-degree-seq}. The labels $(a,b,c)$ that we use for the vertices correspond to their coordinates with respect to the base $B$.

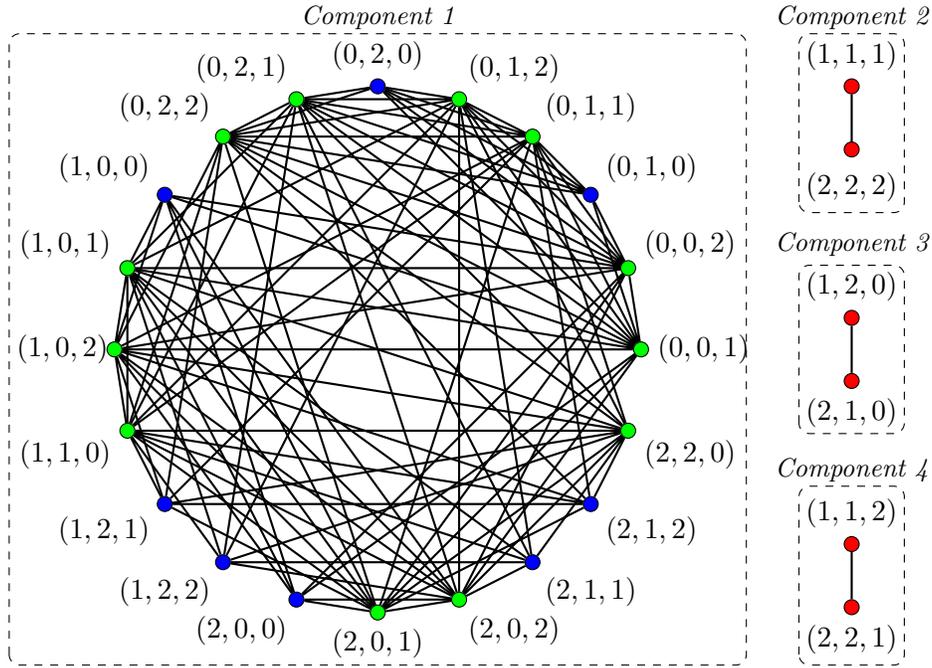
\begin{figure}[H]
\begin{center}
\begin{tikzpicture}[scale=0.7, every node/.style={inner sep=2pt}]

\node[draw=none,rectangle,inner sep=1pt,font=\small] at (9,6.3) {Component 2};
\draw[dashed,rounded corners] (8,2.6) rectangle (10,6);
\node[circle,draw,fill=red,minimum size=3pt] (v25) at (9,5) {};
\node[font=\scriptsize,inner sep=0pt,label={[label distance=1.2mm] $(1,1,1)$}] at (9,5) {};
\node[circle,draw,fill=red,minimum size=3pt] (v12) at (9,3.8) {};
\node[font=\scriptsize,inner sep=0pt,label={[label distance=1.2mm] $(2,2,2)$}] at (9,2.5) {};

 \node[draw=none,rectangle,inner sep=1pt,font=\small] at (9,2) {Component 3};
 \draw[dashed,rounded corners] (8,-1.6) rectangle (10,1.6);
   \node[circle,draw,fill=red,minimum size=3pt] (v14) at (9,0.6) {};
 \node[font=\scriptsize,inner sep=0pt,label={[label distance=1.2mm] $(1,2,0)$}] at (9,0.6) {};
  \node[circle,draw,fill=red,minimum size=3pt] (v20) at (9,-0.6) {};
  \node[font=\scriptsize,inner sep=0pt,label={[label distance=1.2mm] $(2,1,0)$}] at (9,-1.8) {};

 \node[draw=none,rectangle,inner sep=1pt,font=\small] at (9,-2.3) {Component 4};
 \draw[dashed,rounded corners] (8,-6) rectangle (10,-2.6);
 \node[circle,draw,fill=red,minimum size=3pt] (v13) at (9,-3.7) {};
 \node[font=\scriptsize,inner sep=0pt,label={[label distance=1.2mm] $(1,1,2)$}] at (9,-3.7) {};
 \node[circle,draw,fill=red,minimum size=3pt] (v24) at (9,-4.9) {};
 \node[font=\scriptsize,inner sep=0pt,label={[label distance=1.2mm] $(2,2,1)$}] at (9,-6.1) {};

  \node[circle,draw,fill=green] (v0) at (5,0) {};
  \node[font=\scriptsize,inner sep=0pt,label={[label distance=1.2mm] 0.0: $(0,0,1)$}] at (5,0) {};
  
  \node[circle,draw,fill=green,minimum size=3pt] (v1) at (4.755,1.545) {};
  \node[font=\scriptsize,inner sep=0pt,label={[label distance=1.2mm] 18.0: $(0,0,2)$}] at (4.755,1.545) {};
  
  \node[circle,draw,fill=blue,minimum size=3pt] (v2) at (4.045,2.939) {};
  \node[font=\scriptsize,inner sep=0pt,label={[label distance=1.2mm] 36.0: $(0,1,0)$}] at (4.045,2.939) {};
  
  \node[circle,draw,fill=green,minimum size=3pt] (v3) at (2.939,4.045) {};
  \node[font=\scriptsize,inner sep=0pt,label={[label distance=1.2mm] 54.0: $(0,1,1)$}] at (2.939,4.045) {};
  
  \node[circle,draw,fill=green,minimum size=3pt] (v4) at (1.545,4.755) {};
  \node[font=\scriptsize,inner sep=0pt,label={[label distance=1.2mm] 72.0: $(0,1,2)$}] at (1.545,4.755) {};
  
  \node[circle,draw,fill=blue,minimum size=3pt] (v5) at (0,5) {};
  \node[font=\scriptsize,inner sep=0pt,label={[label distance=1.2mm] 90.0: $(0,2,0)$}] at (0,5) {};
  
  \node[circle,draw,fill=green,minimum size=3pt] (v6) at (-1.545,4.755) {};
  \node[font=\scriptsize,inner sep=0pt,label={[label distance=1.2mm] 108.0: $(0,2,1)$}] at (-1.545,4.755) {};
  
  \node[circle,draw,fill=green,minimum size=3pt] (v7) at (-2.939,4.045) {};
  \node[font=\scriptsize,inner sep=0pt,label={[label distance=1.2mm] 126.0: $(0,2,2)$}] at (-2.939,4.045) {};
  
  \node[circle,draw,fill=blue,minimum size=3pt] (v8) at (-4.045,2.939) {};
  \node[font=\scriptsize,inner sep=0pt,label={[label distance=1.2mm] 144.0: $(1,0,0)$}] at (-4.045,2.939) {};
  
  \node[circle,draw,fill=green,minimum size=3pt] (v9) at (-4.755,1.545) {};
  \node[font=\scriptsize,inner sep=0pt,label={[label distance=1.2mm] 162.0: $(1,0,1)$}] at (-4.755,1.545) {};
  
  \node[circle,draw,fill=green,minimum size=3pt] (v10) at (-5,0) {};
  \node[font=\scriptsize,inner sep=0pt,label={[label distance=1.2mm] 180.0: $(1,0,2)$}] at (-4.8,0) {};
  
  \node[circle,draw,fill=green,minimum size=3pt] (v11) at (-4.755,-1.545) {};
  \node[font=\scriptsize,inner sep=0pt,label={[label distance=1.2mm] -162.0: $(1,1,0)$}] at (-4.755,-1.545) {};
  
  \node[circle,draw,fill=blue,minimum size=3pt] (v15) at (-4.045,-2.939) {};
  \node[font=\scriptsize,inner sep=0pt,label={[label distance=1.2mm] -144.0: $(1,2,1)$}] at (-4.045,-2.939) {};
  
  \node[circle,draw,fill=blue,minimum size=3pt] (v16) at (-2.939,-4.045) {};
  \node[font=\scriptsize,inner sep=0pt,label={[label distance=1.2mm] -126.0: $(1,2,2)$}] at (-2.939,-4.045) {};
  
  \node[circle,draw,fill=blue,minimum size=3pt] (v17) at (-1.545,-4.755) {};
  \node[font=\scriptsize,inner sep=0pt,label={[label distance=1.2mm] -108.0: $(2,0,0)$}] at (-1.545,-4.755) {};
  
  \node[circle,draw,fill=green,minimum size=3pt] (v18) at (-0,-5) {};
  \node[font=\scriptsize,inner sep=0pt,label={[label distance=1.2mm] -90.0: $(2,0,1)$}] at (-0,-5) {};
  
  \node[circle,draw,fill=green,minimum size=3pt] (v19) at (1.545,-4.755) {};
  \node[font=\scriptsize,inner sep=0pt,label={[label distance=1.2mm] -72.0: $(2,0,2)$}] at (1.545,-4.755) {};
  
  \node[circle,draw,fill=blue,minimum size=3pt] (v21) at (2.939,-4.045) {};
  \node[font=\scriptsize,inner sep=0pt,label={[label distance=1.2mm] -54.0: $(2,1,1)$}] at (2.939,-4.045) {};
  
  \node[circle,draw,fill=blue,minimum size=3pt] (v22) at (4.045,-2.939) {};
  \node[font=\scriptsize,inner sep=0pt,label={[label distance=1.2mm] -36.0: $(2,1,2)$}] at (4.045,-2.939) {};
  
  \node[circle,draw,fill=green,minimum size=3pt] (v23) at (4.755,-1.545) {};
  \node[font=\scriptsize,inner sep=0pt,label={[label distance=1.2mm] -18.0: $(2,2,0)$}] at (4.755,-1.545) {};

 \draw[thick] (v25) -- (v12)  (v14) -- (v20) (v13) -- (v24);
  

\draw[dashed,rounded corners] (-7,-6) rectangle (7,6);
\node[draw=none,rectangle,inner sep=1pt,font=\small] at (0,6.3) {Component 1};

\draw[thick] (v0) -- (v1) (v0) -- (v2) (v0) -- (v3) (v0) -- (v4)  (v0) -- (v5) (v0) -- (v6) (v0) -- (v7) (v0) -- (v8) (v0) -- (v9)  (v0) -- (v10) (v0) -- (v17)  (v0) -- (v18) (v0) -- (v19);
  
\draw[thick] (v1) -- (v2) (v1) -- (v3)  (v1) -- (v4)  (v1) -- (v5)  (v1) -- (v6) (v1) -- (v7) (v1) -- (v8) (v1) -- (v9) (v1) -- (v10) (v1) -- (v17) (v1) -- (v18)  (v1) -- (v19);
  
\draw[thick] (v2) -- (v3) (v2) -- (v4)  (v2) -- (v5) (v2) -- (v6)  (v2) -- (v7);
  
\draw[thick] (v3) -- (v4)  (v3) -- (v5)  (v3) -- (v6) (v3) -- (v7) (v3) -- (v10) (v3) -- (v11) (v3) -- (v15) (v3) -- (v18) (v3) -- (v22) (v3) -- (v23);
  
\draw[thick] (v4) -- (v5)  (v4) -- (v6)  (v4) -- (v7) (v4) -- (v9) (v4) -- (v11) (v4) -- (v16) (v4) -- (v19) (v4) -- (v21)  (v4) -- (v23);
  
\draw[thick] (v5) -- (v6)  (v5) -- (v7)  (v6) -- (v7) (v6) -- (v9) (v6) -- (v11) (v6) -- (v16) (v6) -- (v19)  (v6) -- (v21) (v6) -- (v23);
  
\draw[thick] (v7) -- (v10)  (v7) -- (v11) (v7) -- (v15) (v7) -- (v18)  (v7) -- (v22) (v7) -- (v23);
  
\draw[thick] (v8) -- (v9)   (v8) -- (v10) (v8) -- (v17) (v8) -- (v18) (v8) -- (v19) (v9) -- (v10)  (v9) -- (v11)  (v9) -- (v16)  (v9) -- (v17) (v9) -- (v18)  (v9) -- (v19) (v9) -- (v21)  (v9) -- (v23);
  
\draw[thick] (v10) -- (v11)  (v10) -- (v15) (v10) -- (v17) (v10) -- (v18) (v10) -- (v19) (v10) -- (v22)  (v10) -- (v23);
  
\draw[thick] (v11) -- (v15)  (v11) -- (v16) (v11) -- (v18)  (v11) -- (v19) (v11) -- (v21) (v11) -- (v22) (v11) -- (v23);
 
\draw[thick] (v15) -- (v18)  (v15) -- (v22) (v15) -- (v23)  (v16) -- (v19) (v16) -- (v21)  (v16) -- (v23)  (v17) -- (v18)  (v17) -- (v19);
  
\draw[thick] (v18) -- (v19) (v18) -- (v22)  (v18) -- (v23)  (v19) -- (v21) (v19) -- (v23)  (v21) -- (v23) (v22) -- (v23);
\end{tikzpicture}
\quad
\caption{The solvable graph $\Gamma_\mathfrak{S}(\mathfrak{sl}_2(\F_3))$}
\label{sl2F3-solvable-graph}
\end{center}
\end{figure}
\end{example}

The graph has $26$ vertices and decomposes into four connected components. The three degree classes visible in the picture reflect the spectral type of the matrices: those with no eigenvalues, one eigenvalue, or two eigenvalues. This matches the decomposition predicted by Theorem \ref{sl2-degree-seq}.

\quad

Using SageMath and GAP we get the following table:
 
\begin{table}[H]
\centering
\caption{Degree sequences of $\Gamma_{\mathfrak S} (\mathfrak{sl}_2(\F_q))$}
\renewcommand{\arraystretch}{1.25}
\begin{tabular}{c|l}
\hline
$q$ & Degree sequence \\ \hline 
$3$  & $\underbrace{13,\ldots, 13}_{12}$, $\underbrace{7,\ldots, 7}_{8}$, $\underbrace{1,\ldots, 1}_{6}$\\[2pt] \hline
$5$  & $\underbrace{43,\ldots, 43}_{60}$, $\underbrace{23,\ldots, 23}_{24}$, $\underbrace{3,\ldots, 3}_{40}$\\[2pt] \hline
$7$  & $\underbrace{89,\ldots, 89}_{168}$, $\underbrace{47,\ldots, 47}_{48}$, $\underbrace{5,\ldots, 5}_{126}$\\[2pt] \hline
$11$ & $\underbrace{229,\ldots, 229}_{660}$, $\underbrace{119,\ldots, 119}_{120}$, $\underbrace{9,\ldots, 9}_{550}$\\ \hline
\vdots  & \vdots \\ 
\end{tabular}
\end{table}


\quad

\begin{theorem}\label{sl2-degree-seq}
For $q>2$ the degree sequence of $\Gamma_{\mathfrak S}\,(\mathfrak{sl}_2(\F_q))$ is:
\begin{equation*}
(\underbrace{2q^{2}-q-2}_{\frac{q(q^2-1)}{2}-\text{times}}, \ \underbrace{q^{2}-2}_{(q^2-1)}, \ \underbrace{q-2}_{\frac{q(q-1)^2}{2}}),
\end{equation*}
and the graph is non-connected for $q>2$.
\end{theorem}



\begin{proof}  
Consider the algebra acting on itself. A typical element of $L=\mathfrak{sl}_2(\F_q)$ can be represented by a matrix of the form $\left(\begin{matrix} a&b\\c&-a\end{matrix}\right)$, where $(a,b,c)\neq (0,0,0)$. This matrix has characteristic polynomial $\lambda^2-a^2-bc$. There are three possibilities for eigenvalues: there can be none, one or two. We'll consider each of these possibilities in turn.
\medskip

\noindent {\bf No eigenvalues:} This occurs if $a^2+bc$ is not a square. Each of $a,b,c$ can take q values, except that each of $b,c$ must be non-zero, so there are $q(q-1)^2$ possibilities for $a^2+bc$. However, in $\F_q$, precisely half of them will be quadratic residues, so there are $\frac{q(q-1)^2}{2}$ vertices represented by elements of this kind.
\par

If $x$ is such a vertex, it can't belong to a two-dimensional subalgebra. Hence, it can only be attached to $\alpha x$ where $\alpha\neq 0,1$. This leaves $q-2$ possibilities.
\medskip

\noindent {\bf One eigenvalue:} This occurs when $a^2+bc=0$. If $c=0$, then $a=0$ and there are $q-1$ choices for $b$. Similarly, if $b=0$, there are $q-1$ choices for $c$. If $b,c\neq 0$, there are $(q-1)^2$ choices for $bc$, half of which are quadratic residues. For each quadratic residue, there are two choices for $a$, so there are $(q-1)^2$ possibilities here. The total number of such vertices, therefore, is $(q-1)^2+2(q-1)=q^2-1$.
\par

If $x$ is such a vertex, it belongs to precisely one two-dimensional subalgebra, spanned by $x,y$, say. Then $x$ can only be attached to $\alpha x+\beta y$ where $(\alpha, \beta)\neq (0,0), (1,0)$. This leaves $q^2-2$ possibilities.
\medskip

\noindent {\bf Two eigenvalues:} This total number of vertices in the graph is $q^3-1$, so the number remaining for this case is $$q^3-1-\frac{q(q-1)^2}{2}-(q^2-1)=\frac{(q-1)(q^2+q)}{2}=\frac{q(q^2-1)}{2}.$$
\par

If $x$ is such a vertex, it belongs to precisely two two-dimensional subalgebras $M_1, M_2$ spanned by $x,y$ and $x,w$, respectively. It can be attached to any vertex in $M_1$, $\alpha x+\beta y$, other than those for which $(\alpha,\beta)=(0,0),(1,0)$, and there are $q^2-2$ of these, as in the previous case. It can also be attached to any vertex in $M_2$, $\gamma x+ \delta w$, other than $0$ and those already counted; that is, for which $\beta=0$, and there are $q^2-q$ of these. Hence, there are $q^2-2+q^2-q=2q^2-q-2$ possibilities.
\medskip

None of the vertices corresponding to the first case can be connected to any of the vertices in the other two cases, so the graph is non-connected.
\end{proof}

This result shows that the solvable graph of $\mathfrak{sl}_2(\F_q)$ breaks into multiple components, with vertices grouped according to the number of eigenvalues of the corresponding matrix (none, one, or two). In particular, vertices with no eigenvalues are only linked to scalar multiples of themselves, producing isolated components. This contrasts with the group-theoretic setting, where solvable graphs are always connected, and highlights a genuinely new Lie-theoretic phenomenon.

\quad

Using SageMath and GAP we get the following table:
 
\begin{table}[H]
\centering
\caption{Degree sequences of $\Gamma_{\mathfrak S}(\mathfrak{gl}_2(\F_q))$}
\renewcommand{\arraystretch}{1.25}
\begin{tabular}{c|l}
\hline
$q$ & Degree sequence \\ \hline
$3$  & $\underbrace{41,\ldots, 41}_{36}$, $\underbrace{23,\ldots, 23}_{24}$, $\underbrace{5,\ldots, 5}_{18}$\\[2pt] \hline
$5$  & $\underbrace{219,\ldots, 219}_{300}$, $\underbrace{119,\ldots, 119}_{120}$, $\underbrace{19,\ldots, 19}_{200}$\\[2pt] \hline

$7$  & $\underbrace{629,\ldots, 629}_{1176}$, $\underbrace{335,\ldots, 335}_{336}$, $\underbrace{41,\ldots, 41}_{882}$\\[2pt] \hline

$11$ & $\underbrace{2529,\ldots, 2529}_{7260}$, $\underbrace{1319,\ldots, 1319}_{1320}$, $\underbrace{109,\ldots, 109}_{6050}$\\ \hline
\vdots & \vdots \\ 
\end{tabular}
\end{table}



\begin{theorem}\label{Th4.4}
For $q>2$ the degree sequence of $\Gamma_{\mathfrak S}\,(\mathfrak{gl}_2 (\F_q))$ is 
\begin{equation*}
(\underbrace{2q^3-q^2-q-1}_{q^2(q^2-1)/2-\text{times}}, \underbrace{q^3-q-1}_{q^3-q}, \underbrace{q^2-q-1}_{q^2(q-1)^2/2})
\end{equation*}
\end{theorem}

\quad

\begin{proof}  
Write $L= \mathfrak{gl}_2(\F_q)= \mathfrak{sl}_2(\F_q)\oplus Z(L)$, where $Z(L)$ is the centre of $L$, that is, the scalar matrices in $L$.  Then every element of $x\in L$ belongs to a solvable subalgebra of dimension at least two, namely $\F x+Z(L)$. 
\par

Suppose first that $x$ belongs to a solvable subalgebra of dimension two, but not three. Then it is of the form $y + z$ where $y$ belongs to a solvable subalgebra of dimension at most one in $sl_2(\F_q)$ and $z \in Z(L)$. Using Table 2 we have that there are $q^2(q-1)^2/2$ such $x$. Such an $x$ can only be connected to any vertex of the form $\alpha x + \beta z$ provided that $(\alpha,\beta) \neq (0,\beta), (1,0)$ and thus to $q^2-q-1$ other vertices.
\par

Next, suppose that $x$ belongs to just one solvable subalgebra of dimension three. Then it is of the form $y + z$ where $y$ belongs to a solvable subalgebra of dimension at most two in $sl_2(\F_q)$ and $z \in Z(L)$. Using Table 2, we have that there are $q(q^2-1)$ such $x$. Such an $x$ can only be connected to any vertex of the form $\alpha x+\beta y + \gamma z$ provided that $(\alpha,\beta,\gamma) \neq (0,0,\beta), (1,0,0)$ and thus to $q^3-q-1$ other vertices.
\par

Finally, suppose that $x$ belongs to more than one solvable subalgebra of dimension three. Then it is of the form $y+z$ where $y$ belongs to more than one solvable subalgebra of dimension at most two in $\mathfrak{sl}_2(\F_q)$ and $z \in Z(L)$. 
Using Table 2, we have that there are $q^2(q^2-1)/2$ such $x$. If $x$ is such a vertex, it belongs to two three-dimensional subalgebras $M_1, M_2$ spanned by $x,y,z$ and $x,w,z$, respectively. It can be attached to any vertex in $M_1$, $\alpha x+\beta y+\gamma z$, other than those for which $(\alpha,\beta,\gamma)=(0,0,\gamma),(1,0,0)$, and there are $q^3-q-1$ of these. It can also be attached to any vertex in $M_2$, $\delta x+\epsilon w +\zeta z$, other than those already counted; that is, for which $\epsilon =0$, and there are $q^3-q^2$ of these. Hence, there are $q^3-q-1+q^3-q^2=2q^3-q^2-q-1$ possibilities.
\end{proof}

This formula reveals that the solvable graph of $\mathfrak{gl}_2(\F_q)$ naturally separates into three classes of vertices, according to the dimension of the smallest solvable subalgebra they belong to. The three different degree values reflect the different ``spectral types'' that appear inside $\mathfrak{sl}_2(\F_q)$, shifted by the presence of the one-dimensional centre of $\mathfrak{gl}_2(F_q)$. 
Thus, degree distributions provide a combinatorial fingerprint of the underlying Lie algebra structure.

\quad


Finally, we present the complement of the soluble graph and an example of it.  This complement condition immediately implies, for example, that the non-soluble graph of $\Gamma_\mathfrak{S}(\mathfrak{sl}_2(\F_q))$ is connected. 

\begin{definition}
Let $L$ be a finite-dimensional Lie algebra. The non-solvable graph of $L$, denoted by $\Gamma_{\mathfrak{S}}(L)^c$, is a simple undirected graph whose vertex set is $L\setminus \sol(L)$, and two vertices $x$ and $y$ are adjacent if and only if $\langle x, y\rangle$ is a non-solvable subalgebra of $L$.    
\end{definition}

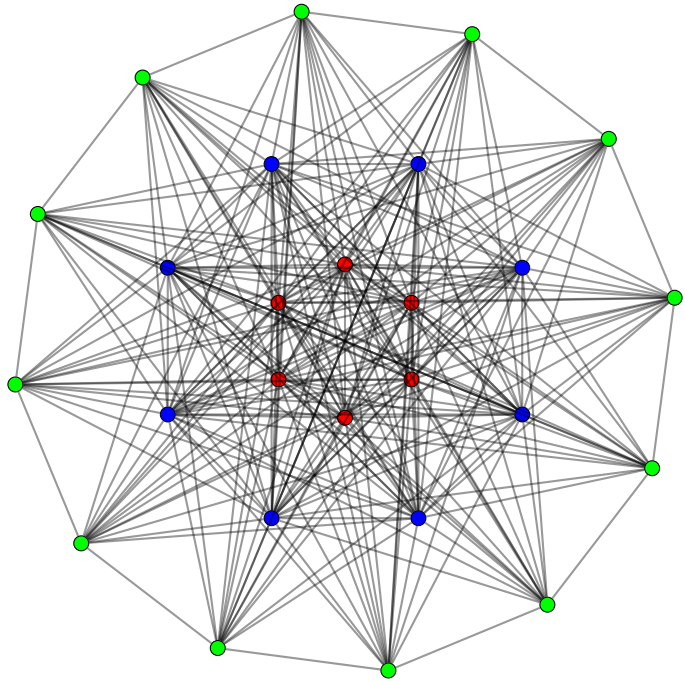
\begin{figure}[H]
    \centering
\tikzset{  v/.style={circle, draw, inner sep=2pt, minimum size=5pt},
  edg/.style={draw, thick, opacity=0.4},
  every label/.style={font=\scriptsize},}
\begin{tikzpicture}[scale=1.7]

\def\Rin{0.6}
\def\Rmid{1.5}
\def\Rout{2.6}

\def\OffIn{0}         
\def\OffMid{22.5}     
\def\OffOut{7.5}      

\foreach \i [evaluate=\i as \ang using {\OffIn+90+60*\i}] in {1,...,6} {
  \node[v, fill=red, label={[label distance=2pt]\ang: }] (c\i) at (\ang:\Rin) {};
}
\foreach \i [evaluate=\i as \ang using {\OffMid+45+45*\i}] in {1,...,8} {
  \node[v, fill=blue, label={[label distance=2pt]\ang: }] (m\i) at (\ang:\Rmid) {};
}
\foreach \i [evaluate=\i as \ang using {\OffOut+30+30*\i}] in {1,...,12} {
  \node[v, fill=green, label={[label distance=2pt]\ang:  }] (o\i) at (\ang:\Rout) {};
}

\foreach \i in {1,...,6} {
  \foreach \j in {1,...,6} {
    \ifnum\i<\j
      \draw[edg] (c\i) -- (c\j);
    \fi
  }
}
\foreach \i in {1,...,6} {
  \foreach \j in {1,...,8} {
    \draw[edg] (c\i) -- (m\j);
  }
}
\foreach \i in {1,...,6} {
  \foreach \j in {1,...,12} {
    \draw[edg] (c\i) -- (o\j);
  }
}
\foreach \i in {1,...,8} {
  \foreach \j in {1,...,8} {
    \ifnum\i<\j
      \draw[edg] (m\i) -- (m\j);
    \fi
  }
}
\foreach \i in {1,...,8} {
  \foreach \j in {1,...,12} {
    \draw[edg] (m\i) -- (o\j);
  }
}
\foreach \i in {1,...,12} {
  \pgfmathtruncatemacro{\next}{mod(\i,12)+1}
  \draw[edg] (o\i) -- (o\next);
}

\end{tikzpicture}

\caption{The non-solvable graph $\Gamma_\mathfrak{S}(\mathfrak{sl}_2(\F_3))$}
    \label{fig:placeholder}
\end{figure}

In contrast with Figure~1, the complement graph is connected, showing that the non-solvable structure of $\mathfrak{sl}_2(\F_3)$ is more tightly linked than the solvable one. Green, blue,  and red vertices correspond to the three degree classes described in Example \ref{Example4.2}.

\section{Algorithms}

\begin{algorithm}[H]
\caption{An algorithm to generate the solvable and non-solvable graphs of a finite  matrix Lie algebra $L$}
    \begin{algorithmic}
        \State $L \gets \text{Lie Algebra in GAP}$
        \State $elements \gets L.\text{AsList}()$
        \State $Sol \gets \text{solvabilizer}(L)$
        \State $nodes \gets [x \in elements \mid x \notin Sol]$
        \Statex
        \State $G \gets \text{Graph}()$
        \State $G. \text{add\_vertices}(nodes)$
        \Statex
        
        \For{$i \gets 0$ to $\text{length}(nodes)$}
            \For{$j \gets i + 1$ to $\text{length}(nodes)$}
                \State $S \gets L.\text{Subalgebra}([nodes[i], nodes[j]])$
                \If{$S.\text{Issolvable}()$}
                    \State $G.\text{add\_edge}(nodes[i], nodes[j])$
                \EndIf
            \EndFor
        \EndFor
        \Statex
        \State \# Show the graph
        \State $G.\text{show}(\text{figsize}=[10, 10], \text{vertex\_labels}=True)$
        \Statex
        \State \# Connected components
        \State $components \gets \text{G.\text{connected\_components}()}$
        \State \textbf{Print:} $G.\text{is\_connected()}$
        \Statex
        \State \# Graph's complement 
        \State $G_{\text{complement}} \gets G.\text{complement()}$
        \State $G_{\text{complement}}.\text{show}(\text{figsize}=[5, 5], \text{vertex\_labels}=True)$
        \State \textbf{Print:} $G_{\text{complement}}.\text{is\_connected()}$
    \end{algorithmic}
\end{algorithm}


    

\bibliographystyle{plane}

\end{document}